\documentclass{amsart}
\usepackage{amsmath}
\usepackage{amssymb}
\usepackage{amsthm}
\usepackage{verbatim}
\usepackage{graphicx}
\usepackage{graphics}
\usepackage{color}
\usepackage{enumerate}
\usepackage{mathrsfs}
\usepackage{booktabs}
\usepackage[toc,page]{appendix}
\usepackage{enumerate}
\usepackage{mathrsfs}
\usepackage{fancyhdr}
\usepackage{latexsym, amsfonts, amssymb, amsmath,  amsthm}
\usepackage{amsopn, amstext, amscd,pifont}
\usepackage{amssymb,bm, cite}
\usepackage[all]{xy}
\usepackage{color}
\usepackage{graphicx,hyperref}
\usepackage{float}
\usepackage{listings}
\usepackage{setspace}
\usepackage{url}
\usepackage{paralist}

\usepackage{graphicx,pgf}
\usepackage{float}

\usepackage{fancyhdr}

\theoremstyle{plain}
\newtheorem{theorem}{Theorem}[section]
\newtheorem{proposition}[theorem]{Proposition}
\newtheorem{lemma}[theorem]{Lemma}
\newtheorem{corollary}[theorem]{Corollary}

\theoremstyle{definition}

\newtheorem{example}[theorem]{Example}

\newtheorem{question}[theorem]{Question}

\graphicspath{{images/}}

\begin{document}

\title{Indeterminacy Loci of Iterate Maps}
\author{Hongming Nie}
\address{Indiana University, 831 East Third Street, Rawles Hall, Bloomington, Indiana 47405, U.S.A. }
\email{nieh@indiana.edu}
\date{\today}
\maketitle

\begin{abstract}
We consider the indeterminacy locus $I(\Phi_n)$ of the iterate map $\Phi_n:\overline{M}_d\dashrightarrow\overline{M}_{d^n}$, where $\overline{M}_d$ is the GIT compactification of the moduli space $M_d$ of degree $d$ complex rational maps. We give natural conditions on $f$ that imply $[f]\in I(\Phi_n)$. These provide partial answers to a question of Laura DeMarco in \cite{De2}
\end{abstract}
\section{Introduction}
For $d\ge 2$, let $\mathrm{Rat}_d$ be the space of degree $d$ complex rational maps $f:\mathbb{P}^1\to\mathbb{P}^1$, thought of as dynamical systems. Parameterizing by the coefficients, $\mathrm{Rat}_d$ is a dense subset of $\mathbb{P}^{2d+1}$. For any point $f\in\mathbb{P}^{2d+1}$, we can uniquely write $f=H_f\hat f=[H_fF_a:H_fF_b]$, where $\hat f$ is a rational map of degree at most $d$. Define 
$$I(d)=\{f=H_f\hat f\in\mathbb{P}^{2d+1}:\hat f=[a:b]\in\mathbb{P}^1\ \text{and}\ H_f(a,b)=0\}.$$
Then the indeterminacy loci of iterate maps $\Psi_n:\mathbb{P}^{2d+1}\dashrightarrow\mathbb{P}^{2d^n+1}$, sending $f$ to $f^n$, are $I(d)$ for all $n\ge 2$ \cite[Theorem 0.2]{De1}.\par 
The moduli space of degree $d$ complex rational maps is defined by the quotient space $M_d:=\mathrm{Rat}_d/\mathrm{PGL}_2(\mathbb{C})$ under the action by conjugation. Then the element $[f]\in M_d$ is the conjugacy class of $f\in\mathrm{Rat}_d$. The moduli space $M_d$ admits a compactification $\overline{M}_d$ . This is constructed as a GIT quotient of a larger semistable locus $\mathrm{Rat}_d^{ss}$ containing $\mathrm{Rat}_d$ under the action of $\mathrm{PGL}_2(\mathbb{C})$; see\cite{Si1,Si2} for definitions. In even degrees, then semistable locus $\mathrm{Rat}_d^{ss}$ coincides with the related stable locus $\mathrm{Rat}_d^s$; in odd degrees, this is no longer the case. For each $n\ge 2$, the iterate map $\Phi_n:M_d\to M_{d^n}$, sending $[f]$ to $[f^n]$, induces a rational map $\Phi_n:\overline{M}_d\dashrightarrow\overline{M}_{d^n}$. However, the iterate maps $\Phi_n$ are not regular for any $d\ge 2$ and $n\ge 2$ on $\overline{M}_d$ \cite[\S 10]{De2}. It is natural to investigate the indeterminacy locus $I(\Phi_n)$. For rational maps on varieties and indeterminacy loci, we refer \cite{Harris}.\par 
In \cite{De2}, DeMarco proved if $f\not\in I(d)$ and $f^n$ is stable for some $n\ge 2$, then $[f]\not\in I(\Phi_n)$ \cite[Lemma 4.2]{De2}. She posed the following questions.
\begin{question}\cite[Question]{De2}\label{question}
\begin{enumerate}
\item If $f\in I(d)$ is stable, is $[f]\in I(\Phi_n)$? 
\item If $f\not\in I(d)$ is stable but $f^n$ is not stable, is $[f]\in I(\Phi_n)$?
\end{enumerate}
\end{question}
In this short paper, we first give an affirmative answer to Question \ref{question} (1).
\begin{theorem}\label{thm-not-ind}
For any $d\ge 2$, if $f\in I(d)$ is stable, then $[f]\in I(\Phi_n)$.
\end{theorem}
We next give a partial result for Question \ref{question} (2).
Assume $f=H_f\hat f$ satisfies the conditions in Question \ref{question} (2). Then $\deg\hat f \ge 1$, see Proposition \ref{not-constant}. Furthermore, if $d$ is odd and $f^n$ is semistable, then $\Phi_n$ is well-defined at $[f]$. Proposition \ref{odd} claims there exist such $f$ and $n$. For the even case, if $d=2$, then $\deg\hat f=1$. By Corollary \ref{deg2} and \cite[Theorem 5.1]{De2}, we know the answer is yes for Question \ref{question} (2). More general, we prove
\begin{theorem}\label{thm-deg1}
For even $d\ge 6$, suppose $f=H_f\hat f\not\in I(d)$ is stable and $\deg\hat f=1$. If $f^n$ is not stable, then $[f]\in I(\Phi_n)$.
\end{theorem}
For $f=H_f\hat f\in\mathbb{P}^{2d+1}$, a zero of $H_f$ in $\mathbb{P}^1$ is called to be a hole of $f$. We say $f$ has a hole orbit relation if there exist two holes $h_1$ and $h_2$ (not necessary distinct) of $f$ such that $\hat f^n(h_1)=h_2$ for some $n\ge 1$. For even $d$, if $f\not\in I(d)$ is stable but $f^n$ is not stable, then $f$ has holes orbit relations. Indeed, consider the probability measure $\mu_f$ defined in \cite{De1}. If $f\in \mathrm{Rat}_d^s$ had no hole orbit relations, then $\mu_f(\{z\})\le 1/2$ for all $z\in\mathbb{P}^1$. By \cite[Proposition 3.2]{De2}, $f^n$ would be stable for all $n\ge 1$. The main idea to prove Theorem \ref{thm-deg1} is to weaken the hole orbit relations, in the sense, to construct $f_t$ by perturbing $f$ such that $\mu_{f_t}(\{z\})\le 1/2$ for all $z\in\mathbb{P}^1$. The assumption $\deg\hat f=1$ gives us convenient normal forms for $f$, see Proposition \ref{conjugate}. The assumption $d\ge 6$ allows us to obtain $f_t$ by only perturbing the holes of $f$.

\section{Some Properties}
Recall the indeterminacy locus $I(d)\subset\mathbb{P}^{2d+1}$. If $f\in\mathbb{P}^{2d+1}\setminus I(d)$, then $f^n$ has the formula \cite[Lemma 2.2]{De1} 
$$f^n=(\prod_{k=0}^{n-1}(H_f\circ\hat f^k)^{d^{n-k-1}})\hat f^n.$$
Denote by $\mathrm{Hole}(f)$ the set consisting of the holes of $f$. For $h\in\mathrm{Hole}(f)$, the multiplicity of $h$ as a zero of $H_f$ is its depth, denoted by $d_h(f)$. Let $m_h(\hat f)$ be the local degree of $\hat f$ at $h$. Set $m_h(\hat f)=0$ if $\hat f$ is a constant.
\begin{lemma}\cite[Lemma 2.4]{De2}\label{depth}
If $f\in\mathbb{P}^{2d+1}\setminus I(d)$, then for all $z\in\mathbb{P}^1$
$$d_z(f^n)=d^{n-1}d_z(f)+\sum_{k=1}^{n-1}d^{n-1-k}m_z(\hat f^k)d_{\hat f^k(z)}(f).$$
\end{lemma}
More general, for the composition map, we have
\begin{lemma}\label{composition}\cite[Lemma 2.6]{De2}
The composition map 
$$\mathcal{C}_{d,e}:\mathbb{P}^{2d+1}\times\mathbb{P}^{2e+1}\dashrightarrow\mathbb{P}^{2de+1}$$
sending a pair $(f,g)$ to the composition $f\circ g$ is continuous away from
$$I(d,e)=\{(f,g)=(H_f\hat f,H_g\hat g):\hat g=c\ \text{and}\ H_f(c)=0\}.$$
Furthermore, for each $(f,g)\in\mathbb{P}^{2d+1}\times\mathbb{P}^{2e+1}$ such that $\deg\hat g>0$,
$$d_z(f\circ g)=d\cdot d_z(g)+m_z(\hat g(z))(f).$$
\end{lemma}
The next lemma states the relations between (semi)stability and depths of holes.
\begin{lemma}\cite[\S 3]{De2}\label{stability}
Let $f=H_f\hat f\in\mathbb{P}^{2d+1}$.
\begin{enumerate}
\item For even $d\ge 2$, $f\in \mathrm{Rat}_d^s$ if and only if 
\begin{enumerate}
\item $d_h(f)\le d/2$ for each hole $h$, and 
\item if $d_h(f)=d/2$, then $\hat f(h)\not=h$.
\end{enumerate}
\item For odd $d\ge 3$, $f\in \mathrm{Rat}_d^s$ if and only if 
\begin{enumerate}
\item $d_h(f)\le(d-1)/2$ for each hole $h$, and 
\item if $d_h(f)=(d-1)/2$, then $\hat f(h)\not=h$,
\end{enumerate}
\item For odd $d\ge 3$, $f\in\mathrm{Rat}_d^{ss}$ if and only if 
\begin{enumerate}
\item $d_h(f)\le(d+1)/2$ for each hole $h$, and 
\item if $d_h(f)=(d+1)/2$, then $\hat f(h)\not=h$.
\end{enumerate}
\end{enumerate}
\end{lemma}
Now we show if $f=H_f\hat f$ satisfies the assumptions in Question \ref{question} (2), then $\hat f$ is nonconstant. 
\begin{proposition}\label{not-constant}
If $f=H_f\hat f\in\mathrm{Rat}_d^s\setminus I(d)$ but $f^n\not\in\mathrm{Rat}_{d^n}^s$, then $\deg\hat f\ge 1$.
\end{proposition}
\begin{proof}
We prove the case when $d$ is even. Similar argument works for the case when $d$ is odd.\par
Suppose $\hat f\equiv c\in\mathbb{P}^1$. Since $f\in\mathrm{Rat}_d^s\setminus I(d)$, then $c$ is not a hole of $f$ and for any hole $h$ of $f$, we have $d_{h}(f)\le d/2$. Note $f^n=H_f^{d^{n-1}}c$. Thus the depth of each hole of $f^n$ has depth at most $d^n/2$ and $c$ is not a hole of $f^n$. Thus $f^n\in\mathrm{Rat}_{d^n}^s$.
\end{proof}
To end this section, we consider the case when $d$ is odd.
\begin{proposition}\label{odd}
Assume $d$ is odd. Suppose $f=H_f\hat f\in\mathrm{Rat}_d^s\setminus I(d)$ with $\deg\hat f= 1$. If $f^n\not\in\mathrm{Rat}_{d^n}^s$, then $f$ is conjugate to $X^{(d-1)/2}Y^{(d-1)/2}[aY:X]$ for some $a\in\mathbb{C}\setminus\{0\}$. Moreover, $n$ is even and $f^n\in\mathrm{Rat}_{d^n}^{ss}\setminus\mathrm{Rat}_{d^n}^s$.
\end{proposition}
\begin{proof}
Since $f^n\not\in\mathrm{Rat}_{d^n}^s$, by Lemma \ref{stability}, there exists a hole $h$ of $f^n$ such that $d_h(f^n)\ge(d^n-1)/2$. By Lemma \ref{depth}, we have 
$$d_h(f^n)=d^{n-1}d_h(f)+\sum_{k=1}^{n-1}d^{n-1-k}d_{\hat f^k(h)}(f).$$
Thus $d_h(f)=(d-1)/2$ and $d_{\hat f^k(h)}(f)=(d-1)/2$ for $1\le k\le n-1$. \par 
Since $f\in\mathrm{Rat}_d^s$, $\hat f(h)\not=h$. Note $\hat f(h)$ is a hole of $f$ with depth $(d-1)/2$. Then $h$ and $\hat f(h)$ are the only holes of $f$. Furthermore, we know $\hat f(\hat f(h))=h$. Now we conjugate $h$ to $0$ and $\hat f(h)$ to $\infty$. Then we get $f$ is conjugate to 
$$f_1([X:Y])=X^{\frac{d-1}{2}}Y^{\frac{d-1}{2}}[aY:X]$$
for some $a\in\mathbb{C}\setminus\{0\}$.\par 
Note 
$$
f_1^n([X:Y])=
\begin{cases}
X^{\frac{d^n-1}{2}}Y^{\frac{d^n-1}{2}}[aY:X]\ &\text{if $n$ is odd},\\
X^{\frac{d^n-1}{2}}Y^{\frac{d^n-1}{2}}[X:Y]\ &\text{if $n$ is even}.
\end{cases}
$$
By Lemma \ref{stability}, if $n$ is odd, $f_1^n\in\mathrm{Rat}_{d^n}^s$, hence $f^n\in\mathrm{Rat}_{d^n}^s$. Therefore, $n$ is even. Again by Lemma \ref{stability}, $f_1^n\in\mathrm{Rat}_{d^n}^{ss}\setminus\mathrm{Rat}_{d^n}^s$, hence $f^n\in\mathrm{Rat}_{d^n}^{ss}\setminus\mathrm{Rat}_{d^n}^s$.
\end{proof}

\section{Proof of Theorem \ref{thm-not-ind}}
First note Lemma \ref{stability} implies the following lemma, which allows us to assume $d\ge 4$. 
\begin{lemma}\label{deg4}
For $d\ge 2$, then
$\mathrm{Rat}_d^s\cap I(d)=\emptyset$ if and only if $d=2$ or $3$.
\end{lemma}
The following lemma gives us the lower bound of the number of holes for $f\in\mathrm{Rat}_d^s\cap I(d)$
\begin{lemma}\label{num-hole}
For $d\ge 4$, if $f\in\mathrm{Rat}_d^s\cap I(d)$, then $f$ has at least $3$ holes.
\end{lemma}
\begin{proof}
Let's prove the case $d$ is even. Similar argument works for the case $d$ is odd. Suppose $f$ has two holes, $h_1, h_2$ with depths $d_1,d_2$. Write $f=H_f\hat f$. Then $\deg H_f=d$ and $d_1+d_2=d$. Since $f\in\mathrm{Rat}_d^s$, by Lemma \ref{stability}, $d_1\le d/2$ and $d_2\le d/2$. Hence $d_1=d_2=d/2$. Since $f\in I(d)$, $\hat f$ is a constant which is a hole, say $h_1$. Again by Lemma \ref{stability}, $d_1< d/2$ since $f\in \mathrm{Rat}_d^s$. It is a contradiction. 
\end{proof}
Now we can prove Theorem \ref{thm-not-ind}
\begin{proof}[Proof of Theorem \ref{thm-not-ind}] 
We prove the case when $d$ is even. The perturbations $\{g_t\}$ and $\{h_t\}$ constructed in the proof still work for the case when $d$ is odd.\par 
By Lemmas \ref{deg4} and \ref{num-hole}, we can assume $d\ge 4$ and normalize by conjugation that $f$ has holes at $0,1$ and $\infty$,
$$f([X:Y])=H_f(X,Y)\hat f([X:Y])=X^{d_0}Y^{d_\infty}(X-Y)^{d_1}\prod_{i=2}^k(X-c_iY)^{d_i}[0:1],$$
where $1\le d_1\le d/2$, $1\le d_i<d/2$ for $i\in\{0,2,\cdots,k,\infty\}$ and $c_2,\cdots, c_k$ are distinct points in $\mathbb{C}\setminus\{0,1\}$. \par 
For $t\in\mathbb{C}\setminus\{0\}$, set
$$g_t([X:Y])=H_{f}(X,Y)[t:1]$$
and 
$$h_t([X:Y])=\frac{H_{f}(X,Y)}{Y}[tX:Y].$$
Then $g_t$ and $h_t$ are stable but not in $I(d)$ for sufficiently small $t\not=0$. Moreover, $g_t$ and $h_t$ converges to $f$ as $t\to 0$. Hence $[g_t]$ and $[h_t]$ converge to $[f]$ as $t\to 0$\par 
Note 
$$g_t^n([X:Y])=(H_f(X,Y))^{d^{n-1}}[t:1]$$
and 
$$h_t^n([X:Y])=\prod_{m=0}^{n-1}\left(\frac{H_f(t^mX,Y)}{Y}\right)^{d^{n-1-m}}[t^mX:Y].$$
Then for sufficiently small $t\not=0$, by Lemma \ref{stability}, $g_t^n$ and $h_t^n$ are stable. Set $g_n=\lim\limits_{t\to 0}g_t^n$ and $h_n=\lim\limits_{t\to 0}h_t^n$.
Then we have 
$$g_n([X:Y])=(H_f(X,Y))^{d^{n-1}}[0:1]$$
and 
$$h_n([X:Y])=X^{\frac{d^n-1}{d-1}d_0}Y^{d^{n-1}d_\infty-\frac{d^{n-1}-1}{d-1}d_0}\left(\frac{H_f(X,Y)}{X^{d_0}Y^{d_\infty}}\right)^{d^{n-1}}[0:1].$$
Note $d_0(g_n)<d^n/2$, $d_0(h_n)<d^n/2$ and the depths of all other holes of $g_n$ and $h_n$ are $\le d^n/2$. Thus $g_n$ and $h_n$ are stable. Thus, as $t\to 0$, $\Phi_n([g_t])$ converges to $[g_n]$ and $\Phi_n([h_t])$ converges to $[h_n]$. However, $[g_n]\not=[h_n]$ since $d_0(g_n)\not=d_0(h_n)$. Thus $[f]\in I(\Phi_n)$ for all $n\ge 2$.
\end{proof}

\section{Proof of Theorem \ref{thm-deg1}}
In this section, we prove Theorem \ref{thm-deg1}. From now on, we suppose $d\ge 2$ is even.\par
A degree $1$ rational map can be conjugate to either $z\to z+1$ or $z\to\omega z$ for some $\omega\in\mathbb{C}\setminus\{0\}$ in affine coordinates. Based on this fact, we first give, under the assumptions in Theorem \ref{thm-deg1}, the normal forms of $f$.
\begin{proposition}\label{conjugate}
For any $d\ge 2$, suppose $f=H_f\hat f\in\mathrm{Rat}_d^s\setminus I(d)$ and $\text{deg}\ \hat f=1$. If $f^n$ is not stable, then there exist a degree $d/2-1$ homogeneous polynomial $H(X,Y)$ with $H(1,1)\not=0$ and $\omega\in\mathbb{C}\setminus\{0,1\}$ such that $f$ is conjugate to $(X-Y)^{d/2}H(X,Y)[X+Y:Y]$ or $(X-Y)^{d/2}H(X,Y)[\omega X:Y]$.
\end{proposition}
\begin{proof}
First there is a hole $h\in\mathbb{P}^1$ of $f$ such that $d_h(f)=d/2$. Indeed, if the depth of each hole is $\le d/2-1$, by Lemma \ref{depth}, for any $z\in\mathbb{P}^1$, we have
\begin{align*}
d_z(f^n)&=d^{n-1}d_z(f)+\sum_{k=1}^{n-1}d^{n-1-k}d_{\hat f^n(z)}(f)\\
&\le d^{n-1}(\frac{d}{2}-1)(1+\sum_{k=1}^{n-1}d^{-k})\\
&<\frac{d^n}{2},
\end{align*}
which contradicts to $f^n$ is not stable. Since $\deg\hat f=1$, we have $f$ is conjugate to either $f_1([X:Y])=H_1(X,Y)[X+Y:Y]$ or $f_2([X:Y])=H_2(X,Y)[\omega X:Y]$ for some $\omega\in\mathbb{C}\setminus\{0\}$. Let $h_1,h_2$ be holes of depth $d/2$ for $f_1,f_2$, respectively. By the stability, we know $h_1\in\mathbb{P}^1\setminus\{[1:0]\}$, $h_2\in\mathbb{P}^1\setminus\{[1:0],[0:1]\}$ and $\omega\not=1$. Thus by conjugating further, we can assume $[1:1]$ is the hole of $f_1$ and $f_2$ with depth $d/2$ and get the conjugate formula for $f$.
\end{proof}
\begin{corollary}\label{deg2}
If $d=2$, then under the assumptions in Proposition \ref{conjugate}, $f$ is conjugate to $(X-Y)[\omega X:Y]$, where $\omega\not=1$ is a $q$-th root of unity for some $1<q\le n$.
\end{corollary}
\begin{proof}
First we show $f$ is not conjugate to $f_1([X:Y])=(X-Y)[X+Y:Y]$. Suppose not. Then, by Lemma \ref{depth}, for any $h\in\mathbb{P}^1\setminus\{[1:1]\}$, we have 
$$d_h(f_1^n)\le\sum_{k=1}^{n-1}2^{n-1-k}<2^{n-1}.$$
For $h=[1:1]$, we have 
$d_{[1:1]}(f_1^n)=2^{n-1}$ and $\hat{f}_1^n([1:1])\not=[1:1]$. Thus $f_1^n$ is stable. It is a contradiction.\par 
Then $f$ is conjugate to $(X-Y)[\omega X:Y]$ for some $\omega\in\mathbb{C}\setminus\{0,1\}$. By \cite[Lemma 5.2]{De2}, $\omega$ is a $q$-th root of unity for some $1<q\le n$.
\end{proof}
Now we can prove Theorem \ref{thm-deg1}.
\begin{proof}[Proof of Theorem \ref{thm-deg1}]
According to Lemma \ref{conjugate}, we have two cases.\par
Case I: We first assume 
$$f([X:Y])=H_f(X,Y)\hat f([X:Y])=(X-Y)^\frac{d}{2}H(X,Y)[X+Y:Y],$$
where $H(X,Y)$ is a degree $d/2-1$ homogeneous polynomial with $H(1,1)\not=0$.\par 
Since $f^n$ is not stable, there exists $h\in\mathbb{P}^1$ such that $d_h(f^n)\ge d^n/2$. Note for any $z\not=[1:1]$, by Lemma \ref{depth},
\begin{align*}
d_z(f^n)&\le d^{n-1}d_z(f)+\sum_{k=1}^{n-1}d^{n-1-k}\frac{d}{2}\\
&\le d^{n-1}(\frac{d}{2}-1)+\frac{d^n}{2}\frac{1}{d-1}(1-\frac{1}{d^{n-1}})\\
&<\frac{d^n}{2}.
\end{align*}
Thus $h=[1:1]$. Furthermore, by Lemma \ref{depth}, we have 
$$d_{[1:1]}(f^n)\le\frac{d^n}{2}+\sum_{k=1}^{n-1}d^{n-1-k}(\frac{d}{2}-1)<d^{n}-d^{n-1}.$$
Note $\hat f^n([1:1])\not=[1:1]$. Since $f^n$ is not stable, by Lemma \ref{stability}, $d_{[1:1]}(f^n)>d^n/2.$\par
Consider
$$g_t([X:Y])=H_{g_t}(X,Y)\hat g_t([X:Y])=(X-(1+t)Y)^\frac{d}{2}H(X,Y)[X+Y:Y]$$
and 
$$h_t([X:Y])=(X-(1+t)Y)^{\frac{d}{2}-1}(X-Y)H(X,Y)[X+Y:Y].$$
Then for sufficiently small $t\not=0$, $g_t$ and $h_t$ are stable, and as $t\to 0$, $g_t\to f$ and $h_t\to f$. Set $M_t([X:Y])=[tX+Y:Y]$. Let $g_n=H_{g_n}\hat g_n:=\lim\limits_{t\to 0}M_t^{-1}\circ g_t^n\circ M_t$. Note 
$$M_t^{-1}\circ\hat g^n_t\circ M_t=[tX+nY:tY]\to Y[1:0].$$
Thus $\hat g_n([X:Y])=[1:0]$. Note 
$$\mathrm{Hole}(g_t^n)=\bigcup_{k=0}^{n-1}\hat g_t^{-k}(\mathrm{Hole}(g_t)).$$
Then 
$$\mathrm{Hole}(M_t^{-1}\circ g_t^n\circ M_t)=M_t^{-1}(\bigcup_{k=0}^{n-1}\hat g_t^{-k}(\mathrm{Hole}(g_t))).$$
Since $f^n$ is not stable, there exists a hole relation: there is $1<k_0<n$ such that $\hat f^{k_0}([1:1])$ is a zero of $H(X,Y)$. Since the perturbations are only at holes, $\hat f=\hat g_t$. Thus $[1:1]\in \mathrm{Hole}(g_t^n)$. Conjugation by $M_t^{-1}$ sends holes $[1+t:1],[1:1]$ and $[1:0]$ of $g_t^n$ to holes $[1:1],[0:1]$ and $[1:0]$ of $M_t^{-1}\circ\hat g^n_t\circ M_t$, respectively, and, as $t\to 0$, $M_t^{-1}(a_t)\to [1:0]$ for any $a_t\in\mathrm{Hole}(g_t^n)\setminus\{[1+t:1],[1:1],[1:0]\}$. Thus 
$$\mathrm{Hole}(g_n)=\{[0:1],[1:1],[1:0]\}.$$ 
Moreover, by Lemma \ref{depth}, for sufficiently small $t\not=0$, $d_{[1+t:1]}(g_t^n)=d^n/2$. Thus $d_{[1:1]}(g_n)=d^n/2$. So we have  
$$M_t^{-1}\circ g_t^n\circ M_t\to g_n=(X-Y)^\frac{d^n}{2}X^{d_0}Y^{d_\infty}[1:0],$$
where $d_0\ge 1$. Thus $d_\infty<d^n/2$, hence $g_n$ is stable. \par 
Now we claim 
$$d_{[1:1]}(f^n)=d_0+\frac{d^n}{2}.$$
Indeed, by the construction of $g_t$, we have 
$$d_{[1:1]}(f^n)=d_{[1+t:1]}(g_t^n)+d_{[1:1]}(g_t^n).$$
The claim then follows from the equality $d_{[1:1]}(g_t^n)=d_{[0:1]}(g_n)=d_0$.\par
Similarly, we have 
$$M_t^{-1}\circ h_t^n\circ M_t\to h_n:=(X-Y)^{(\frac{d}{2}-1)d^{n-1}}X^{d'_0}Y^{d'_\infty}[1:0],$$
where $d'_0=d_{[1:1]}(f^n)-(d/2-1)d^{n-1}>d^{n-1}$. Thus $d'_\infty< d^n/2$, hence $h_n$ is stable.\par
If $[g_n]=[h_n]$, there exists $M\in\mathrm{PGL}_2(\mathbb{C})$ such that $M^{-1}\circ g_n\circ M=h_n$ since $g_n$ and $h_n$ are stable. This conjugacy $M$ must send holes of $h_n$ to holes of $g_n$, preserving their depths, and map $[1:0]$ to $[1:0]$.  Note $d_{[1:1]}(g_n)\not=d_{[1:1]}(h_n)$. Then $M([1:1])=[0:1]$. Hence $d_0=(d/2-1)d^{n-1}$. Thus $d_{[1:1]}(f^n)=d^{n}-d^{n-1}$. It is a contradiction. Thus $[g_n]\not=[h_n]$. So $[f]\in I(\Phi_n)$.\par 
Case II: Now we assume 
$$f([X:Y])=(X-Y)^\frac{d}{2}H(X,Y)[\omega X:Y],$$
where $H(X,Y)$ is a degree $d/2-1$ homogeneous polynomial with $H(1,1)\not=0$ and $\omega\in\mathbb{C}\setminus\{0,1\}$. \par 
If $\omega$ is not a $q$-th root of unity for all $1<q\le n$, set
$$g_t([X:Y])=(X-(1+t)Y)^\frac{d}{2}H(X,Y)[\omega X:Y]$$
and 
$$h_t([X:Y])=(X-(1+t)Y)^{\frac{d}{2}-1}(X-Y)H(X,Y)[\omega X:Y].$$
Applying the same argument in Case I, we know $[f]\in I(\Phi_n)$.\par 
If $\omega$ is a $q$-th root of unity for some $1<q\le n$, we may assume $\omega$ is a primitive $q$-th root of unity and $n=kq+r$ for some $0\le r<q$. First note for any $m\ge 1$, by Lemma \ref{depth}, we have for any $z\in\mathbb{P}^1\setminus\{[1:1]\}$ , 
$$d_{z}(f^m)\le d^{m-1}(\frac{d}{2}-1)+\sum_{k=1}^{m-1}d^{m-1-k}\frac{d}{2}<\frac{d^m}{2}$$ and 
$$d_{[1:1]}(f^m)<\frac{d^m}{2}+\sum_{k=1}^{m-1}d^{m-1-k}\frac{d}{2}=\frac{d^{m+1}-d}{2(d-1)}.$$
Set 
$$g_t([X:Y])=(X-(1+t)Y)^{\frac{d}{2}-2}(X-(1-t)Y)(X-(1+2t)Y)H(X,Y)[\omega X:Y]$$
and 
$$h_t([X:Y])=(X-(1+t)Y)^{\frac{d}{2}-1}(X-Y)H(X,Y)[\omega X:Y].$$
Let $M_t([X:Y])=[tX+Y:Y]$. Then 
$$M_t^{-1}\circ g_t^q\circ M_t\to g_q:=(X-Y)^{(\frac{d}{2}-2)d^{q-1}}(X+Y)^{d^{q-1}}(X-2Y)^{d^{q-1}}X^{d_0}Y^{d_\infty}[X:Y],$$
where $d_0=d_{[1:1]}(f^q)-d^q/2$. Thus $d_h(g_q)<d^q/2$ for any $h\in\mathbb{P}^1$. Hence $g_q$ is stable. Similarly, we have 
$$M_t^{-1}\circ h_t^q\circ M_t\to h_q:=(X-Y)^{(\frac{d}{2}-1)d^{q-1}}X^{d'_0}Y^{d'_\infty}[X:Y],$$
where $d'_0=d_{[1:1]}(f^q)-(d/2-1)d^{q-1}$. Note $d^{q-1}\le d'_0< d^q/2$. Thus $d'_\infty< d^q/2$. Hence $h_q$ is stable. Note $g_q$ and $h_q$ are not in $I(d^q)$. So we have 
$$M_t^{-1}\circ g_t^{kq}\circ M_t=(M_t^{-1}\circ g_t^q\circ M_t)^k\to g_q^k$$
and 
$$M_t^{-1}\circ h_t^{kq}\circ M_t=(M_t^{-1}\circ h_t^q\circ M_t)^{k}\to h_q^k.$$
By Lemma \ref{depth}, we know for any $z\in\mathbb{P}^1$, $d_z(g_q^k)< d^{qk}/2$ and $d_z(h_q^k)< d^{qk}/2$. Thus $g_q^k$ and $h_q^k$ are stable.\par
Similarly, we have 
$$M_t^{-1}\circ g_t^{r}\circ M_t\to g_r:=(X-Y)^{(\frac{d}{2}-1)d^{r-1}}(X+Y)^{d^{r-1}}X^{\tilde{d}_0}Y^{\tilde{d}_\infty}[1:0]$$
and 
$$M_t^{-1}\circ h_t^{r}\circ M_t\to h_r:=(X-Y)^{(\frac{d}{2}-1)d^{r-1}}X^{\tilde{d'}_0}Y^{\tilde{d'}_\infty}[1:0].$$
We have for any $z\in\mathbb{P}^1$, $d_z(g_r)\le d^{r}/2$ and $d_z(h_r)\le d^{r}/2$.\par 
If $r\not=0$, by Lemma \ref{composition}, we have 
$$M_t^{-1}\circ g_t^{n}\circ M_t\to g_r\circ g_q^k$$
and 
$$M_t^{-1}\circ h_t^{n}\circ M_t\to h_r\circ h_q^k.$$
Moreover, for any $z\in\mathbb{P}^1$, 
$$d_z(g_r\circ g_q^k)\le(\frac{d^{qk}}{2}-1)d^r+\frac{d^r}{2}<\frac{d^n}{2}$$
and 
$$d_z(h_r\circ h_q^k)\le(\frac{d^{qk}}{2}-1)d^r+\frac{d^r}{2}<\frac{d^n}{2}.$$
Thus $g_r\circ g_q^k$ and $h_r\circ h_q^k$ are stable.\par 
Now let
$$
g_n=
\begin{cases}
g_q^k\ &\text{if $r=0$},\\
g_r\circ g_q^k\ &\text{if $r\not=0$}.
\end{cases}\ 
\text{and} \ 
h_n=
\begin{cases}
h_q^k\ &\text{if $r=0$},\\
h_r\circ g_q^k\ &\text{if $r\not=0$}.
\end{cases}
$$
Since $g_n$ and $h_n$ have different numbers of holes, $[g_n]\not=[h_n]$. Thus $[f]\in I(\Phi_n)$.
\end{proof}

For $d=4$, suppose $f=H_f\hat f$ satisfies assumptions in Theorem $\ref{thm-deg1}$. If, up to conjugation, $\hat f([X:Y])=[X+Y:Y]$, the argument in the proof of theorem $\ref{thm-deg1}$ still shows $[f]\in I(\Phi_n)$. However, if $\hat f$ is conjugate to $[\omega X:Y]$, where $\omega\in\mathbb{C}\setminus\{0,1\}$ is a $q$-th root of unity for some $1<q\le n$, the method of proof of Theorem $\ref{thm-deg1}$, in which only holes are perturbed, breaks down. 
\begin{example}
{\fontfamily{cmr}\selectfont Consider}
$$f([X:Y])=H_f(X,Y)\hat f([X:Y])=(X-Y)^2Y[-X:Y]$$
and take $n=2$.
Note $f\in\mathrm{Rat}_4^s$ but $f^2\not\in\mathrm{Rat}_{16}^s$. Let $g_t$ and $h_t$ be as in the proof of Theorem $\ref{thm-deg1}$. Then 
$$g_2([X:Y])=h_2([X:Y])=X^4(X-Y)^4Y^7[X:Y].$$
So $[g_2]=[h_2]$. In fact, if $f_t=H_{f_t}\hat f_t$ is a perturbation of $f$ by perturbing the holes of $f$, i.e. $\hat f_t=\hat f$. Then $[f^2_t]$ converges to $[g_2]$.\par 
However, in this example, if we perturb the holes of $f$ and $\hat f$ simultaneously, we may get a different limit. Set 
$$f_t([X:Y])=(X-(1+t)Y)(X-Y)[-XY:tX^2+Y^2].$$
and $M_t([X:Y])=[tX+Y:Y]$. Then 
$$M_t^{-1}\circ f_t^2\circ M_t\to f_2:=X^4(X-Y)^4Y^7[X-2Y:Y].$$
We have $[f_2]\not=[g_2]$. Hence $[f]\in I(\Phi_2)$.
\end{example}

\bibliographystyle{siam}
\bibliography{references}

\end{document}